\documentclass[twoside,a4paper,11pt,limits]{amsart}
\usepackage{enumitem}
\usepackage{graphicx}
\usepackage{color}
\usepackage{amsmath,amssymb,amsthm,dsfont}

\usepackage[hyperef]{paper_diening}
\newtheorem{theorem}{Theorem}[section]
\newtheorem{lemma}[theorem]{Lemma}
\newtheorem{corollary}[theorem]{Corollary}
\newtheorem{proposition}[theorem]{Proposition}

\numberwithin{equation}{section}
\newcommand{\meantmp}[2]{#1\langle{#2}#1\rangle}
\newcommand{\mean}[1]{\meantmp{}{#1}}

\newcommand{\RdN}{{\setR^{d\times N}}}
\newcommand{\Rd}{{\setR^{d}}}

\newcommand{\Ocal}{\mathcal{O}}

\newcommand{\diam}{\text{\rm diam}}

\newcommand{\dx}{\ensuremath{\,{\rm d} x}}
\newcommand{\dy}{\ensuremath{\,{\rm d} y}}

\newcommand{\dlambda}{\ensuremath{\,{\rm d} \lambda}}

\newcommand{\Acal}{{\mathcal{A}}}
\begin{document}

\title[Very weak solutions to systems with $p$-growth]{Existence of very weak solutions to elliptic systems of $p$-Laplacian type}\thanks{M.~Bul\'{\i}\v{c}ek's work is supported by the project LL1202  financed by the Ministry of Education, Youth and Sports, Czech Republic and by the University Centre for Mathematical Modelling, Applied Analysis and Computational Mathematics (Math~MAC) S.~Schwarzacher thanks the program PRVOUK~P47, financed by Charles University in Prague. M.~Bul\'{\i}\v{c}ek is a  member of the Ne\v{c}as Center for Mathematical Modeling}

\author[M.~Bul\'{\i}\v{c}ek]{Miroslav Bul\'{\i}\v{c}ek} %{Miroslav Bul\'{\i}\v{c}ek}
\address{Mathematical Institute, Faculty of Mathematics and Physics, Charles University in Prague
Sokolovsk\'{a} 83, 186 75 Prague, Czech Republic}
\email{mbul8060@karlin.mff.cuni.cz}

%\author[L.~Diening]{Lars Diening} %{Lars Diening}
%\address{Institut f\"{u}r Mathematik, Universit\"{a}t Osnabr\"{u}ck, Albrechtstr. 28a, 49076 Osnabr\"{u}ck, Germany}
%\email{ldiening@uni-osnabrueck.de}

\author[S.~Schwarzacher]{Sebastian Schwarzacher}%{Sebastian Schwarzacher}
\address{Department of Mathematical Analysis, Faculty of Mathematics and Physics,  Charles University in Prague,
Sokolovsk\'{a} 83, 186 75 Prague, Czech Republic}
\email{schwarz@karlin.mff.cuni.cz}

\begin{abstract} We study vector valued solutions to non-linear elliptic partial differential equations with $p$-growth. Existence of a solution is shown in case the right hand side is the divergence of a function which is only $q$ integrable, where $q$ is strictly below but close to the duality exponent $p'$. It implies that possibly degenerate operators of $p$-Laplacian type are well posed in a larger class then the natural space of existence. The key novelty here is a refined a priori estimate, that recovers a duality relation between the right hand side and the solution in terms of weighted Lebesgue spaces.
\end{abstract}


\maketitle

\section{Introduction}
\label{S1}

Let $\Omega \subset \mathbb{R}^d$ be a Lipschitz domain and $S:\Omega \times \mathbb{R}^{d\times N}\to \mathbb{R}^{d\times N}$ be a Carath\'{e}odory mapping. We investigate the existence of a very weak solution $u:\Omega\to\setR^N$ with $N\in\setN$ of the system
\begin{equation}\label{EQ}
\begin{aligned}
\divergence S(\cdot ,\nabla u)&= \divergence |f|^{p-2}f&&\textrm{in } \Omega,\\
u&=0 &&\textrm{on } \partial \Omega.
\end{aligned}
\end{equation}
Here we assume growth, coercivity and monotonicity assumptions on $S$ related to the exponent $p\in (1,\infty)$. Explicitly, that there exist constants $C_1,C_2>0$ and $C_3\geq 0$, such that for all $z_1,z_2\in \mathbb{R}^{d\times N}$ and almost all $x\in \Omega$ there holds
\begin{align}
S(x,z_1)\cdot z_1 &\ge C_1|z_1|^p -C_3,\label{A1}\text{ coercivity}\\
|S(x,z_1)| &\le C_2|z_1|^{p-1}+C_3^\frac{p-1}{p},\label{A2}\text{ boundedness}\\
(S(x,z_1)-S(x,z_2))\cdot (z_1-z_2)&\ge 0,\text{ monotonicity}.\label{A3}
\end{align}
The model case is the $p$-Laplacian system
\begin{equation}\label{PL}
\begin{aligned}
\divergence \abs{\nabla u}^{p-2}\nabla u&= \divergence |f|^{p-2}f&&\textrm{in } \Omega,\\
u&=0 &&\textrm{on } \partial \Omega.
\end{aligned}
\end{equation}

It is well known, that if assumptions \eqref{A1}--\eqref{A3} are satisfied and $f\in L^p(\Omega;\RdN)$, then a solution of \eqref{EQ} exists with $\nabla u\in L^p(\Omega)$. It can for instance be shown by monotone operator theory. 
%Moreover, it is unique in case $S$ is strictly monotone\footnote{The field $S$ is strictly monotone, if \eqref{A3} is a strict inequality for all $z_1\neq z_2$.}.
The starting point of our investigations is the following question.

\smallskip
{\em Q: Does for $f\in L^q(\Omega;\RdN)$ and $q\neq p$ a distributional solution $u$ of \eqref{EQ} exist, such that $\nabla u\in L^q(\Omega;\RdN)$?}

\smallskip
Under general assumptions \eqref{A1}--\eqref{A3} the answer to this question is not affirmative for all $q\in (1,\infty)$. This is well known due to the various counterexamples even in the simplest case $p=2$ and $f\equiv 0$. At the end of exponents smaller then $2$ we mention the non-smooth solutions constructed by Serrin~\cite{Ser65}. At the end of exponents larger than $2$ the non-smooth solutions constructed by Ne\v{c}as~\cite{Ne77}. The example of Serrin~\cite{Ser65} implies also that there can be no hope for uniqueness in the large class of $W^{1,q}_0(\Omega;\setR^N)$ with $q\in (1,p)$, unless stricter assumptions are available. 

However, it turns out that closely around $p$ the existence of solutions with natural integrability is true, even under the minimal assumptions \eqref{A1}--\eqref{A3}. 

%Indeed, in the case $p=2$, it was possible to show that there exists an $\epsilon$ depending on $C_1,C_2$ and $\Omega$ alone, such that existence, uniqueness and estimates are available in case $f\in L^q(\Omega)$ and $q\in [2-\epsilon,2+\epsilon]$~\cite{Bul10}. However, it needed the stricter assumption, that $S$ is Lipschitz continuous in $z$. 
%
%The goal of this article is to extend this result to all $p\in (1,\infty)$. 
%We extend the result of~\cite{Bul10} since we omit the assumption of Lipschitz continuity. 
More precisely, we are able to prove the following theorem.
\begin{theorem}\label{T1}
Let $\Omega \subset \mathbb{R}^{d}$ be a bounded Lipschitz domain and $S$ satisfy \eqref{A1}--\eqref{A3}. Then there is an $\epsilon$ depending on $C_1, C_2,d,N$, $p$ and $\Omega$, such that for all $q\in [p-\epsilon,p]$ the following holds.
If $f\in L^q(\Omega; \mathbb{R}^{d\times N})$, then there exists $u\in W^{1,q}_0(\Omega; \mathbb{R}^N)$ which is a distributional solution to \eqref{EQ}.

Moreover, there is a constant $c$ only depending on $C_1,C_2,p,q,d,N,\Omega$ and a constant $c_1$ depending additionally on $C_3$, such that
\begin{align}
\label{apri1}
\norm{\nabla u}_{L^q(\Omega;\RdN)}\leq c\norm{f}_{L^q(\Omega;\RdN)}+c_1.
\end{align}
\end{theorem}
As was mentioned before, this result is optimal with respect to the generality of the assumptions \eqref{A1}--\eqref{A3}. Let us briefly collect what was already known before. 

In case $q>p$, the existence of a solution in $W^{1,p}_0(\Omega)$ is obvious. The integrability improvement follows by an argument known as Gehring's Lemma. With the minimal assumptions \eqref{A1}--\eqref{A3} it was proved in~\cite[Theorem~7.8]{KriMin06}. See also the more classical result~\cite[Theorem 4.1]{GiaGiu82} and for  an overview and more details~\cite{Min10}.
For the $p$-Laplacian, it is known that \eqref{apri1} holds also in the case of large exponents $q\in [p,\infty)$~\cite{Iwa83,CafPer98} and beyond~\cite{DiBMan93,DieKapSch11}. 

In the special case $p=2$, it was possible to show that there exists an $\epsilon$ depending on $C_1,C_2$ and $\Omega$ alone, such that existence, uniqueness and regularity are available in case $f\in L^q(\Omega)$ and $q\in [2-\epsilon,2+\epsilon]$~\cite{Bul12}. However, it needed the stricter assumption, that $S$ is Lipschitz continuous with respect to $z$. Existence, Uniqueness and Regularity for the full range $q\in (1,\infty)$ has recently been shown to hold, in case $p=2$ with $S$ having additional Uhlenbeck type structure~\cite{BulDieSch15}. 
%This result includes and extends the theory for linear systems with uniformly continuous and elliptic coefficients.

In the case of $p\neq 2$ and $q<p$ very little is known about the existence of a distributional solution. The situation is quite delicate, since even for bounded domains the existence of any object of solution is not obvious. The only existence result available related to powers $q<p$ for the p-Laplacian are restricted to the scalar case $N=1$ and to a better structure of the right hand side. More precisely, when the right hand side is a function (or a Radon measure)~\cite{BoGaVa95}. But even for solutions to \eqref{PL} the existence of solutions was not known in the case of $p\neq 2$ and $q<p$. This seems astonishing, since the a priori estimates~\eqref{apri1} are available for solutions to \eqref{PL} in case $\setR^d=\Omega$ ever since the seminal work of Iwaniec~\cite{Iwa92} in 1992. There it is showed, that there is an $\epsilon>0$ that provided a distributional solution to the p-Laplace exists, it holds already \eqref{apri1} for $q\in [p-\epsilon,p]$. Greco, Iwaniec and Sbordone could stretch the existence frame slightly by showing existence provided the right hand side is in the grande Lebesgue space, $f\in L^{p)}(\Omega;\RdN)$~\cite{GreIwaSbo97}. It is a space, that is slightly larger then $L^p(\Omega;\RdN)$ quantified in terms of logarithmic powers. Later the a priori estimates could be extended to the parabolic case in~\cite{KinLew02}.

After having collected all above efforts we conclude that {\em the existence was open ever since 1992 for $q<p$ and could not be closed ever since.}
Therefore our main result is the existence of a distributional solution in case $q\in [p-\epsilon,p]$ and not the estimate \eqref{apri1}, although it is new in this general form.

The key observation is, that the $L^q$ a priori estimate alone is not suitable to establish solutions to non-linear operators.
This is due to the fact that the a priori estimates inherit only weak compactness or less and the only possible way to mach weak compactness and non-liniarities is via convexity. In the setting here it is reflected via the use of the monotonicity, for instance via the Minty method. However, the method seemed lost the moment the weak limit is not a suitable test function anymore. Only recently a new point of view was established, which allowed to regain a duality relation between $f$ and $\nabla u$~\cite{BulDieSch15}. 
The duality is gained by replacing $L^q$ estimates with weighted $L^p_\omega$ estimates, where the weight is chosen in terms of the right hand side $f$. It has to be chosen in such a way that it belongs to the Muckenhoupt class $\mathcal{A}_q$. This is necessary since it is known, that many linear and sub-linear operators including the Laplace and the Maximal operator are bounded restrictively in these Muckenhoupt classes.

The key novelty, which is of its sovereign interest is the a priori estimate in the existence and regularity result below.
\begin{theorem}\label{T2}
Let $\Omega \subset \mathbb{R}^{d}$ be a bounded Lipschitz domain and $S$ satisfy \eqref{A1}--\eqref{A2}. Then there is an $\epsilon$ depending on $\frac{C_2^{p'}}{C_1}$, $p,d,N$ and $\Omega$, such that for all $q\in [p-\epsilon,p]$ the following holds.
If $f\in L^q(\Omega; \mathbb{R}^{d\times N})$, then there exists $u\in W^{1,q}_0(\Omega; \mathbb{R}^N)$ which is a distributional solution to \eqref{EQ}.

Moreover, there is a constant $c$ depending on $C_1,C_2,p,q,d,N$ and $\Omega$, such that\footnote{Here $M$ is the Hardy Littlewood maximal operator that is defined in Section~2.}
\begin{align}
\label{apri2}
\int_{\Omega} \abs{\nabla u}^{p}(M(f+1))^{q-p}\dx \leq c\int_\Omega \abs{f}^q\dx+ c(C_3+1).
\end{align}
\end{theorem}
Let us point out that the above estimate measures $\nabla u$ more accurate in relation to the right hand side. For once we find by \eqref{est1} that \eqref{apri2} implies \eqref{apri1}. But it also implies for instance that non $p$-integrable singularities of $\nabla u$, can only appear in areas, where $f$ is large, quantified by the naturally related weight. We therefore believe, that estimates of the above type will be of increasing importance in the framework of the existence theory in many applications. Moreover, we wish to indicate its potential for numerical analysis, especial its use for adaptive schemes.

After little restrictions the method can easily be applied on unbounded domains.
\begin{corollary}
\label{C1}
Let $\Omega \subset \mathbb{R}^{d}$ be a Lipschitz domain and $S$ satisfy \eqref{A1}--\eqref{A2}, with $C_3=0$. Then there is an $\epsilon$ depending on $\frac{C_2^{p'}}{C_1}$, $p,d,N$ and $\Omega$, such that for all $q\in [p-\epsilon,p]$ the following holds.
If $f\in L^q(\Omega; \mathbb{R}^{d\times N})$, then there exists $u\in D^{1,q}_0(\Omega; \mathbb{R}^N)$\footnote{$D^{1,q}_0(\Omega; \mathbb{R}^N):= \set{ u\in L^1_{\text{loc}}(\omega;\setR^N): \exists u^j\in C^\infty_0(\Omega;\setR^N)\text{ s.t. }\nabla u^j\to\nabla u\text{ in } L^q(\Omega;\setR^d)}.$} which is a distributional solution to \eqref{EQ}.

Moreover, there is a constant $c$ only depending on $C_1,C_2,p,q,d,N$ and $\Omega$, such that
\begin{align}
\label{apri3}
\int_{\Omega}\abs{\nabla u}^q + \abs{\nabla u}^{p}(Mf)^{q-p}\dx \leq c\int_\Omega \abs{f}^q\dx.
\end{align}
\end{corollary}

The structure of the paper is as follows. We introduce the necessary notation in the preliminary below. In Section~\ref{s1} we introduce a truncation method of Sobolev functions, relative to an open set, which is the analytical highlight of this article. In Section~\ref{s2} we deduce the a-priory estimates and in Section~\ref{s3} we prove the existence.

 %
%It seems, that the a
%that some parts of the above theorem have allready been proved. First at all, the case existence of a solution 
%The case investigated her is the so called {\em very weak frame work}. This is the case, when the right hand side $\divergence |f|^{p-2}f$ is not in the dual of the natural existence space $W^{1,p}_0(\Omega)$.
%We focus on the case $f\in L^q$ for $q\in (p-\varepsilon, p)$ 
%where $\varepsilon$ should depend only on $C_1$, $C_2$ and $\Omega$.
%\begin{theorem}\label{T1}
%Let $\Omega \subset \mathbb{R}^{d}$ be Lipschitz and $S$ satisfy \eqref{A1}--\eqref{A2}. Then there exists $q$ fulfilling $\max(1,(p-1))<q<p$ depending only on $C_1$, $C_2$ $p$ and $\Omega$ such that for any  $f\in L^q(\Omega; \mathbb{R}^{d\times N})$, there exists $u\in W^{1,q}_0(\Omega; \mathbb{R}^N)$ satisfying \eqref{EQ} in a weak sense.
%\end{theorem}
\section{Preliminary}
%\seb{This is copy past, one should throw anything out we do not need.}
Throughout
the paper all cubes will have sides parallel to the axes. By $c,C$ we denote a generic constant, i.e. its value may change at every appearance. Its dependencies are either stated in the results or are indicated in  $C(\dots)$.
%We use the following notation of mean values on bounded measurable sets $E$.
%\[
%\textrm{with} \quad \mean{\abs{f}}_Q:=\dashint_{Q}\abs{f(y)}\dy:=\frac{1}{|Q|}\int_{Q}\abs{f(y)}\dy.
%\] 
Let us recall the definition of the Hardy Littlewood maximal function. For any $f\in L^1_{\loc}(\setR^n)$ we define
\[
Mf(x):=\sup_\set{Q\text{ a cube }:\,x\in Q}\mean{\abs{f}}_Q \quad \textrm{with} \quad \mean{\abs{f}}_Q:=\dashint_{Q}\abs{f(y)}\dy:=\frac{1}{|Q|}\int_{Q}\abs{f(y)}\dy.
\]
 It is standard, that the operator is sub linear and continuous from $L^s(\Omega)\to L^s(\Omega)$, for $s\in (1,\infty]$. 
% We shall use  similar notation for the vector- or tensor-valued function as well. Note here, that we could replace balls in the definition of the maximal function by cubes with sides parallel to the axis without any change. We will also use in what follows the standard notion for Lebesgue and Sobolev spaces. 
Further,  we say that $\omega:\setR^d \to \setR$ is a weight function if it is a measurable function that is almost everywhere finite and positive. For such a weight and arbitrary measurable $\Omega\subset \setR^d$ we denote the space $L^p_{\omega}(\Omega;\setR^N)$ with $p\in [1,\infty)$ as
$$
L^p_{\omega}(\Omega,\setR^d):=\biggset{f:\Omega\to\setR^N,\text{ measurable } : \; \norm{f}_{L^p_\omega}
:= \bigg(\int_{\Omega} |f(x)|^p\omega(x)\dx\bigg)^{\frac 1p} <\infty}.
$$
%Note that our weights are defined on the whole space~$\setR^d$.
We introduce the weighted Sobolev spaces, as 
\begin{align*}
{W}^{1,q}_{\omega}(\Omega;\setR^N)&:=\set{u\in W^{1,1}(\Omega;\setR^N):\nabla u\in L^q_\omega(\Omega;\setR^N)}\\
{W}^{1,q}_{0,\omega}(\Omega;\setR^N)&:=W^{1,1}_0(\Omega;\setR^N)\cap {W}^{1,q}_{\omega}(\Omega;\setR^N).
\end{align*}
%We also introduce the homogeneous Sobolev spaces, in case of unbounded domains. 
%$
%D^{1,q}_\omega(\Omega;\setR^N)= \overline{C^\infty_0(\overline{\Omega},\setR^N)}^{\norm{\nabla \cdot}_{L^q_{\omega}(\Omega)}}
%$
%and
%$
%D^{1,q}_{0,\omega}(\Omega;\setR^N)= \overline{C^\infty_0({\Omega},\setR^N)}^{\norm{\nabla \cdot}_{L^q_{\omega}(\Omega)}}.
%$
%\seb{include the in-variance under constances?}
 Next, for  $p\in [1,\infty)$, we say that a weight $\omega$
belongs to the Muckenhoupt class $\Acal_p$ if and only if
there exists a positive constant $A$ such that for every cube $Q
\subset \setR^n$ the following holds
\begin{alignat}{2}
  \label{defAp2}
  \bigg(\dashint_Q
    \omega\dx\bigg)\bigg(\dashint_Q\omega^{-(p'-1)}\dx\bigg)^\frac{1}{p'-1}
  &\le A & \qquad\qquad&\text{if $p \in (1,\infty)$},
  \\
  \label{defA1}
  M\omega(x)&\le A\, \omega(x) &&\text{if $p=1$}.
\end{alignat}
In what follows, we denote by
$A_p(\omega)$ the smallest constant $A$ for which the
inequality~\eqref{defAp2}, resp.~\eqref{defA1}, holds.

The next result makes a very useful link between the maximal operator and  $\Acal_p$-weights.
\begin{lemma}[See pages 229--230 in \cite{86:_real} and page 5 in \cite{Tur00}]\label{cor:dual}
Let $f \in L^1_{\loc}(\setR^n)$ be such that $Mf<\infty$ almost everywhere in $\setR^n$. Then for all $\alpha \in (0,1)$ we have $(Mf)^{\alpha} \in \Acal_1$. Furthermore, for all $p\in (1,\infty)$ and all $\alpha\in (0,1)$ there holds
$(Mf)^{-\alpha(p-1)}\in \Acal_p$.
\end{lemma}
%
%We would also like  to point out that the maximum $\omega_1 \vee \omega_2$
%and minimum $\omega_1 \wedge \omega_2$ of two $\Acal_p$-weights are again
%$\Acal_p$-weights. For $p=2$ we even have $A_2(\omega_1 \wedge
%\omega_2) \leq A(\omega_1) +A_2(\omega_2)$, which follows from the following simple computation
%\begin{align}
%  \label{eq:A2min}
%  \begin{aligned}
%    \dashint_B \omega_1 \wedge \omega_2\dx \dashint_B \frac1{\omega_1
%      \wedge \omega_2}\dx &\leq \left[\bigg(\dashint_B \omega_1\dx\bigg)
%    \wedge \bigg( \dashint_B\omega_2\dx \bigg)\right] \dashint_B
%    \frac{1}{\omega_1} + \frac{1}{\omega_2}\dx
%    \\
%    &\leq A_2(\omega_1) + A_2(\omega_2).
%  \end{aligned}
%\end{align}

\section{Relative Truncation}
\label{s1}
In this section we introduce a relative truncation. It is strongly influenced by the Lipschitz truncation method, developed in \cite{DieRuzWol10} see also~\cite{BreDieSch13,Die13paseky} where the concept was refined in the direction which can be adapted here. In contrast to the Lipschitz truncation, the {\em relative truncation} smoothens the function relative to a different independent function, or more explicit, it truncates the gradient on any given open set. 

The first step is a suitable covering.
%According to Lemma~3.1 of Ref.~\cite{DieRuzWol10} there exists a Whitney covering $\set{Q_i}$ of~$\Ocal_\lambda$ with finite intersection. With respect to the covering $\set{Q_i}$
%there exists a partition of unity $\set{\phi_i} \subset
%C^\infty_0(\Rd)$. We collect the properties below
%following sense:
%\begin{enumerate}%[label={\rm (PW\arabic{*})}, leftmargin=*]
%\item\label{itm:whit1} $\bigcup_{i}\frac {1} {2} Q_i \,=\, \Ocal_\lambda$,
%\item\label{itm:whit2} For all $i\in \setN$ we have $8
  %Q_i \subset \Ocal_\lambda$ and $16 Q_i \cap (\Rd\setminus
  %\Ocal_\lambda)\neq \emptyset$. 
%\item \label{itm:whit3} If $ Q_i \cap Q_{j} \neq \emptyset
  %$ then $ \frac 12 r_j\le r_i< 2\, {r_j}$, where $r_i := r_{Q_i}$, the radius of $Q_i$. 
%\item \label{itm:whit4}  At every point at most
  %$120^{d+1}$ of the sets $4Q_i$ intersect.
%\item \label{itm:P1} $\chi_{\frac{1}{2}Q_i}\leq \phi_i\leq
  %\chi_{\frac{2}{3}Q_i}$.
%\item \label{itm:P2} $\sum_j\phi_j=\sum_{j\in A_i}\phi_j=1$ on $Q_i$.
%\item \label{itm:P3} $\abs{\phi_i} + r_i \abs{\nabla \phi_i} \leq c(d)$.
%\end{enumerate}
 %Where for each $Q_i$ we use the definition of $A_i:= \set{ j \,:\, Q_j\cap
  %Q_i\neq\emptyset}$. Observe, that by the previous $\# A_i\le 120 ^{d+1}$ and $r_j \sim
%r_i$ for all $j \in A_i$.
%
We take the covering introduced in~\cite[Proposition 3.17]{DieFro10} and proof some more properties needed for our situation.
\begin{proposition}
\label{cover}
 Let $\Ocal$ be an open subset of $\Rd$, with $\Ocal\neq\setR^d$.
Then there exists a countable family $Q_i$ of closed, dyadic cubes such that
\begin{enumerate}[label=(\alph*)]
\item \label{c1}$\bigcup_i Q_i=\Ocal$ and all cubes $Q_i$ have disjoint interiors.
\item \label{c2}$\diam(Q_i)<\text{dist}(Q_i,\Ocal^c)\leq 4\diam(Q_i)$
\item \label{c3} If $Q_i\cap Q_j\neq \emptyset$, then $\diam(Q_i)\leq 2\diam(Q_j)\leq 4\diam(Q_i)$.
\item \label{c4} For given $Q_i$, there exists at most $4^d-2^d$ cubes $Q_j$ 
touching $Q_i$ (boundaries intersect but not the interiors), we define $A_i$ as the index set of all neighboring cubes of $Q_i$.
\item \label{c5} The family of cubes $\set{\frac{3}{2}Q_i}_{i\in\setN}$ has finite intersection. The family can be split in $4^d-2^d$ pairwise disjoint subfamilies.
\item \label{c6}There is a partition of unity, $\psi_i\in C^{0,1}(\frac98Q_i)$, such that $\chi_{\frac12Q_i}\leq \psi_i\leq \chi_{\frac98Q_i}$ and
$\diam(Q_i)\abs{\nabla \psi_i}\leq c(d)$ uniformly.
\end{enumerate}
\end{proposition}
\begin{proof}
Since \ref{c1}--\ref{c4} are stated in~\cite[Proposition 3.17]{DieFro10} we just proof the last two.
First we show that the cubes $\frac{3}{2} Q_i$ have finite intersection. Firstly, by \ref{c2}, we find that $\frac{3}{2} Q_i\subset \Ocal$, which implies bu \ref{c1}, that $\bigcup_i \frac32Q_i=\bigcup_i Q_i$. Secondly, we find by \ref{c3} that $\frac{3}{2} Q_j$ only intersects its neighbors which implies by \ref{c4} that each $x\in Q_i$ is at most covered by $4^d-2^d$ cubes of the family $\frac{3}{2} Q_j$. Moreover, by \ref{c3} we find that the $\frac{9}{8} Q_i$ does not intersect with $\frac12Q_j$, for all $j\neq i$.
Therefore it is standard to construct a partition of unity as requested.
\end{proof}
Next we introduce the {\em relative truncation} of $u\in W^{1,p}_{0}(\Omega;\setR^N)$, with respect to $\Ocal$ an open proper set. Since we do not assume, that $\Ocal\subset \Omega$ we extend $u$ by 0 outside $\Omega$ and define
$$
u_{\Ocal}(x):=\left\{\begin{aligned}
&u(x) &&x\in \mathbb{R}^d \setminus \Ocal,\\
&\sum_i \psi_i \bar{u}_i &&x\in \Ocal.
\end{aligned}
\right.
$$
Where the $\bar{u}_i$ are defined via the covering constructed in Proposition~\ref{cover}:
$$
\bar{u}_i:=\left\{
\begin{aligned}
&\dashint_{\frac98Q_i} u \dx &&\textrm{if } \frac98Q_i \subset \Omega,\\
&0 &&\textrm{otherwise.}
\end{aligned}
 \right.
$$
At first we have to proof the following lemma, which is essential to show that the {\em relative truncation} is stable in Sobolev function spaces.
\begin{lemma}
Let $\Omega$ be a Lipschitz domain and let $u\in W^{1,p}_{0}(\Omega;\setR^N)$, for $p\in [1,\infty)$. Then for any $Q_i,Q_j$ which are members of the covering introduced in Proposition~\ref{cover} with non-empty intersection, we have
\begin{align}
\label{c7}
\frac{\abs{\bar{u}_j-\bar{u}_i}}{\diam(Q_i)}\leq c(d,\Omega)\dashint_{\frac32 Q_i}\abs{\nabla u}\dx+c(d,\Omega)\dashint_{\frac32 Q_j}\abs{\nabla u}\dx
\\
\label{c8}
\dashint_{\frac98Q_i}\Bigabs{\frac{\bar{u}-\bar{u}_i}{\diam(Q_i)}}^p\leq c(d,p,\Omega)\dashint_{\frac32 Q_i}\abs{\nabla u}^p\dx
%+c\dashint_{\frac32 Q_j}\abs{\nabla u}\dx
\end{align}
\end{lemma}
\begin{proof}
The statement of \eqref{c8}, in case $\bar{u}_i\neq0$ is just the \Poincare{} inequality. In case $\bar{u}_i=0$ we enlarge the cube to $\frac32Q_i$. Since we have a Lipschitz boundary and $\frac98Q_i\cap\Ocal\neq \emptyset$ we find that $\abs{\frac32Q_i}\leq c(\Omega)\abs{\frac32Q_i\cap\Omega^c}$. This implies, that we can apply \Poincare's inequality on $u$ for the cube $\frac32Q_i$, which finishes the proof of \eqref{c8}.

To proof \eqref{c7} observe at first, that since $Q_i$ and $Q_j$ are neighbours, we find by \ref{c3} that $C_{i,j}=\frac98Q_i\cap\frac{9}{8}Q_j $ have a comparible maesure, to $\abs{\frac32Q_i}$ and $\abs{\frac32Q_j}$.
Since in case $\bar u_i=\bar u_j=0$, there is nothing to show, let us assume that $\bar u_i\neq0$. By \Poincare{} we find
\begin{align*}
&\abs{\bar{u}_j-\bar{u}_i}\leq \Bigabs{\bar u_j-\dashint_{C_{i,j}}u\dx}+\Bigabs{\dashint_{C_{i,j}}u-\mean{u}_{\frac98Q_i}\dx}\\
&\quad\leq \Bigabs{\bar u_j-\dashint_{C_{i,j}}u\dx}+c(d)\dashint_{\frac98Q_{i}}\abs{u-\mean{u}_{\frac98Q_i}}\dx
\leq \Bigabs{\bar u_j-\dashint_{C_{i,j}}u\dx}+c(d)\diam(Q_i)\dashint_{\frac98 Q_i}\abs{\nabla u}\dx
\end{align*}
In case $\bar u_j\neq 0$ as well, then we find by symmetry
\[
\Bigabs{\bar u_j-\dashint_{C_{i,j}}u\dx}\leq c(d)\diam(Q_j)\dashint_{\frac98 Q_j}\abs{\nabla u}\dx.
\]
In case $\bar u_j=0$ we enlarge the set $C_{i,j}$ to $\frac32 Q_j$ for which we know that $\abs{\frac32Q_j}\leq c(\Omega)\abs{\frac32Q_j\cap\Omega^c}$.
Therefore, finally \Poincare{}'s inequality implies
\[
\dashint_{C_{i,j}}\abs{u}\dx\leq c(d)\dashint_{\frac32Q_j}\abs{u}\dx\leq c(d,\Omega)\diam(Q_j)\dashint_{\frac32Q_j}\abs{\nabla u}\dx.
\]
\end{proof}
This implies, that the {\em relative truncation} is stable in Sobolev spaces.
\begin{lemma}[Stability]
\label{stabil}
Let $u\in W^{1,p}_0(\Omega;\setR^N)$, then $u_\Ocal\in W^{1,p}_0(\Omega;\setR^N)$. Moreover, the following estimate holds
\begin{align*}
&\int_\Omega \abs{\nabla(u-u_{\Ocal})}^p\dx\leq c(d,p,\Omega)\int_{\Ocal\cap\Omega}\abs{\nabla u}^p\dx.
\end{align*}
\end{lemma}
\begin{proof}
It is enough to show the estimate since the zero-trace follows from the very definition of the {\em relative truncation} immediately. We define by $A_i$ the index set of all $j$, such that $\frac{9}{8}Q_j\cap \frac{9}{8}Q_i\neq \emptyset$. Observe here, that $\# A_j\leq 4^n-2^n$ by \ref{c4}. Next we use \ref{c1}, the fact that we have a partition of unity, \ref{c5}, \ref{c3} and \ref{c6}.
\begin{align*}
&\int \abs{\nabla(u-u_{\Ocal})}^p\dx
=\int_{\Ocal} \abs{\nabla (u-u_{\Ocal})}^p\dx
=\sum_{i}\int_{Q_i}\Bigabs{\nabla u-\sum_{j\in A_i}\nabla\psi_j \bar{u}_j}^p\dx\\
&\leq \sum_{i}\int_{Q_i}\Bigabs{\sum_{j\in A_i}\nabla(u-\psi_j \bar{u}_j)}^p\dx
\leq c(d,p)\sum_{i}\sum_{j\in A_i}\int_{Q_i}\abs{\nabla(u-\psi_j \bar{u}_j)}^p\dx
\\
&\leq c(d,p)\sum_{i}\int_{Q_i}\abs{\nabla u}^p\dx+c(d,p)\sum_{i}\sum_{j\in A_i}\int_{Q_i}\Bigabs{\frac{u-u_j}{\diam(Q_i)}}^p\dx
\end{align*}
Now \ref{c1},\eqref{c7} and \eqref{c8} imply together with Jensen's inequality, \ref{c5} and \ref{c2} that
\begin{align*}
&\int \abs{\nabla(u-u_{\Ocal})}^p\dx
\leq c(d,p)\int_{\Ocal}\abs{\nabla u}^p\dx+c(d,p,\Omega)\sum_{i}\int_{\frac32Q_i}\abs{\nabla u}^p\dx\leq c(d,p)\int_{\Ocal}\abs{\nabla u}^p\dx.
\end{align*}
\end{proof}

\section{Uniform a~priori estimates}
In this section, we will not mention the dependence of the constants on $d,N$ anymore.
\label{s2}
\begin{proposition}
Let $\Omega \subset \mathbb{R}^{d}$ be a bounded Lipschitz domain, $S$ satisfy \eqref{A1}--\eqref{A2} and $h\in L^1(\Omega)$ not identically equal to zero. Then there exists $\varepsilon_0>0$ depending only on $\frac{C_2^{p'}}{C_1}$, $p$ and $\Omega$ such that for any $\epsilon \in (0,\epsilon_0)$ the following holds. If $f\in L^p(\Omega; \mathbb{R}^{d\times N})\cap L^p_{(Mh)^{-\epsilon}}(\Omega; \mathbb{R}^{d\times N})$ and $u\in W^{1,p}_0(\Omega; \mathbb{R}^N)$ satisfies \eqref{EQ} in a weak sense, then $\nabla u \in L^p_{(Mh)^{-\epsilon}}(\Omega; \mathbb{R}^{d\times N})$ as well.

Moreover,
\begin{equation}
\int_{\Omega} \frac{|\nabla u|^p}{(Mh)^{\varepsilon}}\dx \le C(C_1,C_2,p,\Omega) \int_{\Omega} \frac{|f|^p+C_3}{(Mh)^{\varepsilon}}\dx. \label{apriori}
\end{equation}
\end{proposition}
\begin{proof}
First, we extend $u$ and $f$ with zero outside $\Omega$. The function $h$, we will approximate by $g=h\chi_\Omega+\delta$. This implies, that $Mg>\delta$ and therefore $f,\nabla u\in L^p_{(Mg)^{-\epsilon}}(\Omega; \mathbb{R}^{d\times N})$ a priori. Once the estimate is established for $g$ with constants independent of $\delta$ the general result follows by letting $\delta\to 0$ and monotone convergence.
%\[
%\Ocal_\lambda=\set{Mg>\lambda}.
%\]
 %Next, for any open bounded $Q\subset \mathbb{R}^d$ we can find a corresponding Withney covering by open balls $\{B_i\}_{i=1}^{\infty}$ (DETAILS SHOULD BE ADDED) and also the corresponding partition of unity $\{\psi_i\}_{i=1}^{\infty}$. Next, we introduce $u_{Q}\in W^{1,p}_0(\Omega; \mathbb{R}^N)$ by the following formula
For any $\lambda>0$ we define $\Ocal(\lambda):=\{x\in \mathbb{R}^d: \, Mg(x)>\lambda\}$. Since the maximal function is sub-linear it is an open set. In case $\Ocal(\lambda)=\setR^n$, we define $u_{\Ocal(\lambda)}=0$. Else, we are able to construct the {\em relative truncation} $u_{\Ocal(\lambda)}$. Testing \eqref{EQ} with $u_{\Ocal(\lambda)}$ which is possible due to Lemma~\ref{stabil} we get the identity
\begin{equation}\label{E1}
\begin{split}
\int_{\{Mg\le \lambda\}}S\cdot \nabla u \dx &= \int_{\{Mg\le \lambda\}}|f|^{p-2}f\cdot \nabla u\dx + \int_{\{Mg>\lambda \}}|f|^{p-2}f\cdot \nabla u_{{\Ocal(\lambda)}}\dx \\
&\quad-\int_{\{Mg>\lambda\}}S\cdot \nabla u_{{\Ocal(\lambda)}}\dx.
\end{split}
\end{equation}
Next, we focus on the estimates of integrals on the set where $Mg>\lambda$. Consider arbitrary $G\in L^{p'}(\Omega; \mathbb{R}^{d\times N})$ and arbitrary $\alpha \in (0,1)$. We get using \ref{c6}--the property of the partition of unity, \eqref{c7}, \ref{c5} H\"older's inequality and \ref{c2} \begin{equation}\label{E2}
\begin{split}
&\int_{\{Mg>\lambda\}} \abs{G\cdot \nabla u_{{\Ocal(\lambda)}}}\dx \le \sum_i \int_{Q_i}|G| \Bigabs{\sum_{j\in A_i} \bar{u}_j \nabla \psi_j}\dx\\
&= \sum_i \int_{Q_i}|G| \Bigabs{\sum_{j\in A_i} (\bar{u}_j-\bar{u}_i) \nabla \psi_j}\dx\\
&\le c(\Omega)\sum_i\sum_{j\in A_i} |{Q_i}|\bigg(\dashint_{Q_i}|G| \dx\bigg)\bigg( \dashint_{\frac32Q_i}|\nabla u|\dx+\dashint_{\frac32Q_j}|\nabla u|\dx\bigg)\\
&= c(\Omega)\sum_i \sum_{j\in A_i} |{Q_i}|\bigg(\dashint_{Q_i}\frac{|G|}{(Mg)^{\frac{\alpha}{p}}} (Mg)^{\frac{\alpha}{p}}\dx \bigg) \bigg(\dashint_{\frac32Q_i}\frac{|\nabla u|}{(Mg)^{\frac{\alpha}{p'}}}(Mg)^{\frac{\alpha}{p'}}\dx
+\dashint_{\frac32Q_j}\frac{|\nabla u|}{(Mg)^{\frac{\alpha}{p'}}}(Mg)^{\frac{\alpha}{p'}}\dx
\bigg)\\
&\le c(\Omega)\sum_i \sum_{j\in A_i}|{Q_i}|\bigg(\dashint_{Q_i}\frac{|G|^{p'}}{(Mg)^{\frac{\alpha p'}{p}}} \dx \bigg)^{\frac{1}{p'}} \bigg(\bigg(\dashint_{\frac32Q_i}\frac{|\nabla u|^p}{(Mg)^{\frac{\alpha p}{p'}}}\dx\bigg)^{\frac{1}{p}}+\bigg(\dashint_{\frac32Q_i}\frac{|\nabla u|^p}{(Mg)^{\frac{\alpha p}{p'}}}\dx\bigg)^{\frac{1}{p}}\bigg)
\\
&\quad
\times\bigg(\dashint_{5Q_i}(Mg)^{\alpha}\dx \bigg).
\end{split}
\end{equation}
We estimate the last integral on the right hand side using two properties. First, since by Lemma~\ref{cor:dual} $(Mg)^{\alpha}$ is an $\mathcal{A}_1$ Muckenhoupt weight, $M(Mg)^{\alpha} \le C(\alpha,d)(Mg)^{\alpha}$. Second, by \ref{c1} we know that $9Q_i\cap {\Ocal(\lambda)}^c\neq \emptyset$.
Consequently, for some $x_0\in9Q_i\cap {\Ocal(\lambda)}^c$, we have 
$$
\begin{aligned}
\dashint_{5Q_i}(Mg)^{\alpha}\dx &\le c\dashint_{9Q_i}(Mg)^{\alpha}\dx\le cM(Mg)^{\alpha}(x_0)\\
&\le c(\alpha)(Mg)^{\alpha}(x_0)\le c(\alpha)\lambda^{\alpha}.
\end{aligned}
$$
We can use this to estimate \eqref{E2} in the following way. For $i\in \setN$ and $j\in A_i$, we deduce by Young's inequality
\begin{equation}\label{E3a}
\begin{split}
& \lambda^{\alpha}|{Q_i}|\bigg(\dashint_{Q_i}\frac{|G|^{p'}}{(Mg)^{\frac{\alpha p'}{p}}} \dx \bigg)^{\frac{1}{p'}} \bigg(\dashint_{\frac32Q_j}\frac{|\nabla u|^p}{(Mg)^{\frac{\alpha p}{p'}}}\dx\bigg)^{\frac{1}{p}}\\
&\le c(\Omega,\alpha)|{Q_i}|\bigg(\dashint_{Q_i}\frac{\lambda^{\frac{\alpha p'}{p}}|G|^{p'}}{(Mg)^{\frac{\alpha p'}{p}}} \dx  + \dashint_{\frac32Q_j}\frac{\lambda^{\frac{\alpha p}{p'}}|\nabla u|^p}{(Mg)^{\frac{\alpha p}{p'}}}\dx\bigg).
\end{split}
\end{equation}
Therefore \eqref{E2} implies using \ref{c5}, that
\begin{equation}\label{E3}
\begin{split}
&\int_{\{Mg>\lambda\}} \abs{G\cdot \nabla u_{{\Ocal(\lambda)}}}\dx 
\le c(\Omega,\alpha)\int_{\{Mg>\lambda\}\cap\Omega} \frac{\lambda^{\frac{\alpha p'}{p}}|G|^{p'}}{(Mg)^{\frac{\alpha p'}{p}}} \dx  + \frac{\lambda^{\frac{\alpha p}{p'}}|\nabla u|^p}{(Mg)^{\frac{\alpha p}{p'}}}\dx.
\end{split}
\end{equation}
Consequently, we get from \eqref{E1} that
\begin{equation}\label{E8}
\begin{split}
\int_{\{Mg\le \lambda\}}S\cdot \nabla u \dx &\le \int_{\{Mg\le \lambda\}}|f|^{p-1}|\nabla u|\dx + c(\Omega,\alpha)\int_{\{Mg>\lambda \}}\frac{\lambda^{\frac{\alpha p'}{p}}(|f|^p+|S|^{p'})}{(Mg)^{\frac{\alpha p'}{p}}} + \frac{\lambda^{\frac{\alpha p}{p'}}|\nabla u|^p}{(Mg)^{\frac{\alpha p}{p'}}}\dx.
\end{split}
\end{equation}
We set $\overline{(p-1)}:= \min((p-1),(p-1)^{-1})$ and take $\epsilon\in (0,\alpha {\overline{(p-1)}})$. 
We multiply the above inequality by $\lambda^{-1-\varepsilon}$ and integrate over $\lambda \in (0,\infty)$ to deduce
\begin{equation}\label{E9}
\begin{split}
&\int_0^{\infty}\int_{\{Mg\le \lambda\}}\frac{S\cdot \nabla u}{\lambda^{1+\varepsilon}} \dx \dlambda \le \int_0^{\infty}\int_{\{Mg\le \lambda\}}\frac{|f|^{p-1}|\nabla u|}{\lambda^{1+\varepsilon}}\dx \dlambda\\
 &\quad+ c(\Omega,\alpha) \int_0^{\infty}\int_{\{Mg>\lambda \}}\frac{\lambda^{\frac{\alpha p'}{p}-1-\varepsilon}(|f|^p+|S|^{p'})}{(Mg)^{\frac{\alpha p'}{p}}} + \frac{\lambda^{\frac{\alpha p}{p'}-1-\varepsilon}|\nabla u|^p}{(Mg)^{\frac{\alpha p}{p'}}}\dx\dlambda.
\end{split}
\end{equation}
We get on the one hand using the Fubini theorem
\begin{equation*}%\label{E10}
\begin{split}
&\frac{1}{\varepsilon}\int_{\Omega}\frac{S\cdot \nabla u}{(Mg)^{\varepsilon}}\dx=\int_{\Omega}\int_{Mg}^{\infty}\frac{S\cdot \nabla u}{\lambda^{1+\varepsilon}} \dlambda \dx= \int_{0}^{\infty}\int_\set{Mg\le \lambda}\frac{S\cdot \nabla u}{\lambda^{1+\varepsilon}} \dx \dlambda; 
\end{split}
\end{equation*}
Thus on the other hand by \eqref{E9} and the Fubini theorem 
\begin{align*}
&\frac{1}{\varepsilon}\int_{\Omega}\frac{S\cdot \nabla u}{(Mg)^{\varepsilon}}\dx\le \int_0^{\infty}\int_{\{Mg\le \lambda\}}\frac{|f|^{p-1}|\nabla u|}{\lambda^{1+\varepsilon}}\dx \dlambda\\
 &\qquad + c(\Omega,\alpha)\int_0^{\infty}\int_{\{Mg>\lambda \}}\frac{\lambda^{\frac{\alpha p'}{p}-1-\varepsilon}(|f|^p+|S|^{p'})}{(Mg)^{\frac{\alpha p'}{p}}} + \frac{\lambda^{\frac{\alpha p}{p'}-1-\varepsilon}|\nabla u|^p}{(Mg)^{\frac{\alpha p}{p'}}}\dlambda\dx\\
&\quad= \int_{\Omega}\int_{Mg}^{\infty}\frac{|f|^{p-1}|\nabla u|}{\lambda^{1+\varepsilon}} \dlambda 
%\\
 %&\qquad
 + c(\Omega,\alpha)\int_{0}^{Mg}\frac{\lambda^{\frac{\alpha p'}{p}-1-\varepsilon}(|f|^p+|S|^{p'})}{(Mg)^{\frac{\alpha p'}{p}}} + \frac{\lambda^{\frac{\alpha p}{p'}-1-\varepsilon}|\nabla u|^p}{(Mg)^{\frac{\alpha p}{p'}}}\dlambda \dx\\
 &\quad = \frac{1}{\varepsilon}\int_{\Omega}\frac{|f|^{p-1}|\nabla u|}{(Mg)^{\varepsilon}}\dx
 %\\
% &\qquad
 + c(\Omega,\alpha)\int_{\Omega}\frac{1}{\frac{\alpha p'}{p}-\varepsilon}\frac{(|f|^p+|S|^{p'})}{(Mg)^{\varepsilon}}+ \frac{1}{\frac{\alpha p}{p'}-\varepsilon}\frac{|\nabla u|^p}{(Mg)^{\varepsilon}} \dx\\
  &\quad \le \frac{1}{\varepsilon}\int_{\Omega}\frac{|f|^{p-1}|\nabla u|}{(Mg)^{\varepsilon}}\dx + \frac{c(\Omega,\alpha)}{\alpha \overline{(p-1)}  -\varepsilon} \int_{\Omega}\frac{|f|^p+|S|^{p'}+|\nabla u|^p}{(Mg)^{\varepsilon}}\dx.
\end{align*}
By the assumptions \eqref{A1},\eqref{A2} and Young's inequality we deduce
\begin{equation}\label{E11}
\begin{split}
\int_{\Omega}\frac{|\nabla u|^p}{(Mg)^{\varepsilon}}\dx \le \frac{c(p)}{C_1}\int_{\Omega}\frac{|f|^{p}}{(Mg)^{\varepsilon}}\dx + \frac{\epsilon C_2^{p'}c(\Omega,p)}{C_1 \alpha (\overline{(p-1)}  -\varepsilon)} \int_{\Omega}\frac{|f|^p+|\nabla u|^p+C_3}{(Mg)^{\varepsilon}}\dx.
\end{split}
\end{equation}
Thus, for $\epsilon \in (0,\epsilon_0)$, we  can absorb the term with $|\nabla u|^p$ on the right hand side of \eqref{E11} into the left hand side and we get \eqref{apriori} by letting $\delta\to 0$. Here 
\begin{align}
\label{epsilon}
\varepsilon_0= \frac{C_1 \overline{(p-1)}}{C_2^{p'}c(\Omega,p)},
\end{align}
since by this choice we can choose for $\epsilon\in (0,\epsilon_0)$ an $\alpha$ accordingly. Observe, that $\alpha\to1$ (which is the non-stable limit), when $\epsilon\to\epsilon_0$.
\end{proof}
%
%\seb{It is possible to omit the a priori knowledge of $\nabla u\in L^p$. (By integrating up to $K$ only). This would imply, that any solution in the weighted class holds the a priori.} 
%

\section{Existence of a solution}
\label{s3}
The proof is making essential use of the following theorem that can be found in \cite[Theorem~2.6]{BulDieSch15}.

\begin{theorem}[weighted, biting div--curl lemma]\label{T5}
  Let $\Omega\subset \setR^n$ be an open bounded set.
  Assume that for some $p\in (1,\infty)$ and given $\omega \in
  \Acal_p$ we have a sequence of vector-valued measurable functions $(a^k, b^k)_{k=1}^{\infty}:\Omega \to \mathbb{R}^n\times \mathbb{R}^n$ such that
  \begin{equation}
    \sup_{k\in \mathbb{N}}\int_{\Omega} \abs{a^k}^p\omega + \abs{b^k}^{p'}\omega \dx <\infty. \label{bit3}
  \end{equation}
  %\marginpar{L: adjust the proof}%
  Furthermore, assume that for every bounded sequence
  $\{c^k\}_{k=1}^{\infty}$ from $W^{1,\infty}_0(\Omega)$ that fulfills
  $$
  \nabla c^k \rightharpoonup^* 0 \qquad \textrm{weakly$^*$ in }
  L^{\infty}(\Omega)
  $$
  there holds
  \begin{align}
    \lim_{k\to \infty} \int_{\Omega} b^k \cdot \nabla c^k \dx
    &=0, \label{bit4}
    \\
    \lim_{k\to \infty} \int_{\Omega} a^k_i \partial_{x_j} c^k -
    a^k_j \partial_{x_i} c^k \dx &=0 &&\textrm{for all }
    i,j=1,\ldots,n.\label{bit5}
  \end{align}
  Then there exists a subsequence $(a^k,b^k)$ that we do not relabel
  and there exists a non-decreasing sequence of measurable subsets
  $E_j\subset\Omega$ with $|\Omega \setminus E_j|\to 0$ as $j\to
  \infty$ such that
  \begin{align}
  a^k &\rightharpoonup a &&\textrm{weakly in } L^1(\Omega;\mathbb{R}^n), \label{bitfa}\\
  b^k &\rightharpoonup b &&\textrm{weakly in } L^1(\Omega;\mathbb{R}^n), \label{bitfb}\\
  a^k \cdot b^k \omega &\rightharpoonup a \cdot b\, \omega &&\textrm{weakly in } L^1(E_j) \quad \textrm{ for all } j\in \mathbb{N}. \label{bitf}
  \end{align}
\end{theorem}

\smallskip

\begin{proof}[Proof of Theorem~\ref{T2}]
As before, we extend every function by zero outside $\Omega$, without further reference.
We approximate a given $f\in L^q(\Omega;\RdN)$, by $f_k:=\min\set{k,|f|}f/|f|$. Then $f_k\in L^q(\Omega;\RdN)\cap L^\infty(\Omega;\RdN)$, we find $|f^k|\nearrow |f|$ and
\begin{align}\label{cfn}
  f^k \to f &&\textrm{strongly in } L^q(\setR^n; \setR^{n\times N}).
\end{align} 
For $f^k$ we can use the standard monotone operator theory to find a solution $u^k\in W^{1,p}_0(\Omega; \mathbb{R}^{N})$ fulfilling
\begin{equation}
  \int_{\Omega} S(x,\nabla u^k) \cdot \nabla \phi \dx = \int_{\Omega}f^k
  \cdot \nabla \phi \dx   \qquad \textrm{ for all } \phi\in W^{1,p}_0(\Omega; \mathbb{R}^N). \label{wfn}
\end{equation}
Hence, we fix $\epsilon=p-q\in (0,\epsilon_0)$.  Then we find by \eqref{apriori} and the continuity of the maximal function that
\begin{align}
\label{est}
\begin{aligned}
\int_{\Omega} \frac{|\nabla u|^p}{(M(f+1))^{\varepsilon}}\dx &\le C(C_1,C_2,p,\Omega) \int_{\Omega} \frac{|f_k|^p+C_3}{(M(f+1))^{\varepsilon}}\dx\\
&\leq c\int M(f)^q\dx +c C_3\leq c\int \abs{f}^q\dx+cC_3. 
\end{aligned}
\end{align}
Moreover, by Young's inequality for the exponents $\frac{q}{p}+\frac{p-q}{p}=1$,
continuity of the maximal function and the previous we gain
\begin{align}
\label{est1}
\begin{aligned}
\int_{\Omega}|\nabla u|^q\dx&=\int_{\Omega} \frac{|\nabla u|^q (M(f+1))^\frac{(p-q)q}{p}}{(M(f+1))^\frac{(p-q)q}{p}}\dx \\
&\leq 
c\int_{\Omega} \frac{|\nabla u|^p}{(M(f+1))^{\varepsilon}}\dx + c\int_{\Omega}(M(f+1))^q\dx\leq c\int \abs{f}^q\dx+c(C_3+1). 
\end{aligned}
\end{align}
 We define the weight $\omega:\frac{1}{(M(f+1))^\epsilon}$. 
Using the a priori estimate, the reflexivity of the corresponding spaces and the growth assumption~\eqref{A2},  we can pass to a subsequence (still denoted
by $u^k$) such that
\begin{align}
  \label{conn-a}
  u^k &\rightharpoonup u &&\textrm{weakly in } W^{1,q}_0(\Omega; \setR^N),
  \\
  \label{conn-b}
  \nabla u^k &\rightharpoonup \nabla u &&\textrm{weakly in }
  L^p_{\omega}\cap L^{q}(\Omega; \setR^{n\times N}),
  \\
  S(x,\nabla u^k) &\rightharpoonup \overline{S} &&\textrm{weakly in }
 L^p_{\omega}\cap L^{q}(\Omega; \setR^{n\times N})\label{con2}.
\end{align}
% Hence, having \eqref{finaln}--\eqref{estkeyn2}, we can extract a subsequence that we do not relabel such that
% \begin{align}
%   u^n &\rightharpoonup u &&\textrm{weakly in }
%   W^{1,\frac{2}{p}}_0(\Omega; \mathbb{R}^N)\label{con1},
%   \\
%   A(x,\nabla u^n) &\rightharpoonup \overline{A} &&\textrm{weakly in }
%   L^{\frac{2}{p}}(\Omega; \mathbb{R}^{n\times N})\label{con2}
%   \\
%   \nabla u^n \sqrt{w} &\rightharpoonup \nabla u \sqrt{w}
%   &&\textrm{weakly in } L^{2}(\Omega; \mathbb{R}^{n\times
%     N})\label{con3},
%   \\
%   A(x,\nabla u^n)\sqrt{\omega} &\rightharpoonup
%   \overline{A}\sqrt{\omega} &&\textrm{weakly in } L^{2}(\Omega;
%   \mathbb{R}^{n\times N}).\label{con4}
% \end{align}
Hence by \eqref{est1},\eqref{est} and \ref{conn-b} and the weak lower semicontinuity
 we obtain
\begin{align}
  \label{final}
  \int_\Omega \abs{\nabla u}^q+ \abs{\nabla u}^p\omega \dx &\leq
  c\int_\Omega\abs{f}^q\dx +c,
\end{align}
which concludes the a priori estimate.

We still have to show that~$u$ is a distributional solution.
 Using
\eqref{wfn}, \eqref{cfn} and \eqref{con2} it follows that
\begin{align}
  \label{eq:limit}
  \int_\Omega  \overline{S}\cdot\nabla \phi \dx=\int_\Omega f\cdot\nabla \phi
  \dx\qquad\text{ for all }\phi\in C^\infty_0(\Omega;\mathbb{R}^N).
\end{align}
To complete the proof of Theorem~\ref{T2}, it remains to show that
\begin{align}
  \overline{S}(x)&=S(x,\nabla u(x)) \qquad\textrm{in  } \Omega.\label{show22}
\end{align}
To do so, we use
%\footnote{Although Theorem~\ref{T5} is formulated forvector-valued functions, it is a easy extension to use it also for  matrix-valued functions, which is the case here.}
Theorem~\ref{T5}. We denote $a^k:=\nabla u^k$ and $b^k:=S(x,\nabla
u^k)$. By using \eqref{final} and
\eqref{A2}, we find that \eqref{bit3} is satisfied with the weight $\omega$. 
Since, we assume by \eqref{epsilon} that $\epsilon<(p-1)$, we find by Lemma~\ref{cor:dual} that $\omega\in A_p$.
Also the assumption \eqref{bit4} holds, which follows from  \ref{cfn},
\eqref{wfn} and \eqref{eq:limit}. Finally, \eqref{bit5} is valid
trivially since $a^k$ is a gradient. Therefore, Theorem~\ref{T5} can
be applied. Meaning, that we have a non-decreasing sequence of measurable
sets $E_j$, such that $|\Omega \setminus E_j|\to 0$ and
\begin{equation*}%\label{Minty}
  S(x,\nabla u^k) \cdot \nabla u^k \omega \rightharpoonup \overline{S}
  \cdot  \nabla u\, \omega \qquad \textrm{weakly in } L^1(E_j).
\end{equation*}
This is enough, to apply some variant of the Minty trick.
For any $G\in L^{p}_{\omega}(\Omega; \setR^{n\times N})$
%, we have  that $G\, \omega$ and also $S(\cdot,G)\,\omega$ belong to  $L^p_{1/\omega}(\Omega;\setR^{n\times N})$. Using
we get by~\ref{conn-b} and~\ref{con2}
\begin{equation*}%\label{Minty2}
  (S(x,\nabla u^k)-S(x,G)) \cdot (\nabla u^k-G) \,\omega \rightharpoonup
  (\overline{S}-S(x,G)) \cdot (\nabla u-G)\, \omega\quad
  \textrm{weakly in }  L^1(E_j).
\end{equation*}
Due to the monotonicity condition~\eqref{A3} we find that the term on the left hand side is non-negative and consequently its weak limit is non-negative as well; especially
$\int_{E_j}(\overline{S}-S(x,G)) \cdot (\nabla u-G) \,\omega \dx \ge 0$ and
\begin{equation*}
\int_{\Omega}(\overline{S}-S(x,G)) \cdot (\nabla u-G)\, \omega \dx \ge \int_{\Omega\setminus E_j}(\overline{S}-S(x,G)) \cdot (\nabla u-G)\, \omega \dx.
\end{equation*}
Letting $j \to \infty$ the Lebesgue dominated convergence theorem implies using the fact that $|\Omega
\setminus E_j|\to 0$ as $j\to \infty$ we obtain
\begin{equation}
\label{minty3}
\int_{\Omega}(\overline{S}-S(x,G)) \cdot (\nabla u-G)\, \omega \dx \ge 0 \qquad \textrm{for all } G\in L^p_{\omega}(\Omega;\setR^{n\times N}).
\end{equation}
Hence, setting $G:=\nabla u -\delta H$ where $H\in L^{\infty}(\Omega; \setR^{n\times N})$ is an arbitrary function and dividing \eqref{minty3} by $\delta$ implies
\begin{equation*}%\label{Minty4}
  \int_{\Omega}(\overline{S}-S(x,\nabla u - \delta H)) \cdot H\, \omega \dx \ge 0.  %\qquad \textrm{for all } H\in L^{\infty}(\Omega;\setR^{n\times N}).
\end{equation*}
Finally, 
%using the Lebesgue dominated convergence theorem, the assumption \eqref{growth} and the continuity of $A$ with respect to the second variable, we can 
letting $\delta \to 0_+$ implies by the continuity assumption of $S$ and dominated convergence
\begin{equation*}
\int_{\Omega}(\overline{S}-S(x,\nabla u)) \cdot H\, \omega \dx \ge 0 \qquad \textrm{for all } H\in L^{\infty}(\Omega;\setR^{n\times N}).
\end{equation*}
Since $\omega$ is strictly positive almost everywhere in $\Omega$, the relation \eqref{show22} easily follows by choosing e.g.,
$$
H:=-\frac{\overline{S}-S(x,\nabla u)}{1+\abs{\overline{S}-S(x,\nabla u)}}.
$$
Thus $u$ is a distributional solution to \eqref{EQ}.
\end{proof}
\begin{proof}[Proof of Corollary~\ref{C1}]
The proof of the a priori estimate~\eqref{apriori} is exactly the same. Analogous to the estimate of \eqref{est},\eqref{est1} one finds in case $C_3=0$, that
\begin{align}
\label{niceest}
\int_\Omega \abs{\nabla u_k}^q+\abs{\nabla u_k}^{p} (Mf)^{q-p}\dx\leq c\int_\Omega \abs{f}^q.
\end{align}
 The existence proof is also analogous. However, since Theorem~\ref{T5} is only valid on bounded domains one has to use that $\overline{S}$ satisfies \eqref{eq:limit} for an arbitrary bounded subset of $\Omega$. This allows to show that $\overline{S}=S(\nabla u)$ in $\Omega'$. Since it was arbitrarily chosen the existence follows. 
\end{proof}
%\section{Uniqueness}
%Problem $\int_{Q_i} \abs{A(\nabla u_1)-A(\nabla u_2}\dashint_{Q_j}\abs{\nabla (u_1-u_2)}$ is only treatibel by shifted guys....
%Use $(u_1-u_2)_\lambda$, for $p>2$ and duality or better the Bulicek trick.
\bibliographystyle{abbrv} %{amsalpha
\bibliography{lars}
\end{document}